\documentclass[a4paper,12pt]{amsart}
\usepackage{amssymb,amscd,amsmath,latexsym}
\usepackage[dvips]{graphicx}

\calclayout

\theoremstyle{plain}

\newtheorem{theo}{Theorem}[section]

\newtheorem{lemm}[theo]{Lemma}
\newtheorem{coro}[theo]{Corollary}

\theoremstyle{definition}
\newtheorem*{defi}{Definition}

\theoremstyle{remark}

\numberwithin{equation}{section}

\newcommand{\field}[1]{\mathbb{#1}}
\newcommand{\C}{\field{C}}
\newcommand{\R}{\field{R}}
\newcommand{\Z}{\field{Z}}
\newcommand{\N}{\field{N}}

\newcommand{\1}{\mathbf 1}

\DeclareMathOperator{\Hom}{Hom}

\def\S1{\C^*}

\def\1{\mathbf 1}

\begin{document}

\title
[Real generalized Bott manifolds of height $2$]
{Cohomological non-rigidity of generalized real Bott manifolds of height $2$}
\date{\today}

\author[M.~Masuda]{Mikiya Masuda}
\address{Department of Mathematics, Osaka City University, Sugimoto, 
Sumiyoshi-ku, Osaka 558-8585, Japan}
\email{masuda@sci.osaka-cu.ac.jp}

\thanks{The author was partially supported by Grant-in-Aid for
Scientific Research 19204007}

\begin{abstract}
We investigate when two generalized real Bott manifolds 
of height $2$ have isomorphic cohomology rings with $\Z/2$ coefficients
and also when they are diffeomorphic.  It turns out that cohomology 
rings with $\Z/2$ coefficients do not distinguish 
those manifolds up to diffeomorphism in general.  This gives a counterexample 
to the cohomological rigidity problem for real toric manifolds posed in \cite{ka-ma08}. 
We also prove that generalized real Bott manifolds of height 2 are 
diffeomorphic if they are homotopy equivalent. 
\end{abstract}

\maketitle

\section{Introduction}

A toric manifold is a compact smooth toric variety and a real toric manifold 
is the set of real points of a toric manifold.  In \cite{ma-su08} we asked
whether toric manifolds are diffeomorphic if their cohomology 
rings with $\Z$ coefficients are isomorphic as graded rings, which is now 
called {\it cohomological rigidity problem for toric manifolds}.  
No counterexample and some partial affirmative 
solutions are known to the problem (see \cite{ch-ma-su08-2}, \cite{ma-su08}). 
If $X$ is a toric manifold and $X(\R)$ is the real toric manifold associated 
to $X$, then $H^*(X(\R);\Z/2)$ is isomorphic to $H^{2*}(X;\Z)\otimes\Z/2$ 
as graded rings.  
Motivated by this, we posed in \cite{ka-ma08} the following analogous problem. 

\medskip
\noindent
{\bf Cohomological rigidity problem for real toric manifolds.} 
Are two real toric manifolds are diffeomorphic if their cohomology rings with 
$\Z/2$ coefficients are isomorphic as graded rings?   

\medskip
We say that 
{\it cohomological rigidity over $\Z/2$} holds for 
a family of closed smooth manifolds 
if the manifolds in the family are distinguished up to diffeomorphism by their 
cohomology rings with $\Z/2$ coefficients. 
A real Bott manifold is the total space of an iterated $\R P^1$ bundles 
where each $\R P^1$ bundle is the projectivization of a Whitney sum of 
two real line bundles.  A real Bott manifold is not only a real toric manifold but also 
a compact flat riemannian manifold.  We proved in \cite{ka-ma08} (and \cite{masu08}) 
that cohomological rigidity over $\Z/2$ holds for 
the family of real Bott manifolds.  

In this paper we consider real toric manifolds obtained as the total spaces 
of projectivization of Whitney sums of real line bundles over a real projective 
space.  We call such a real toric manifold a {\it generalized real Bott 
manifold of height 2}.  In this paper 
we will investigate when those two manifolds have isomorphic cohomology rings 
with $\Z/2$ coefficients and also when they are diffeomorphic. 
As a result, we will see that cohomological rigidity over $\Z/2$ 
fails to hold for some family of generalized real Bott 
manifolds of height 2, which gives 
a negative answer to the cohomological rigidity problem for real toric 
manifolds above. 
We also prove that generalized real Bott manifolds of height 2 are 
diffeomorphic if they are homotopy equivalent. 

The author thanks Y. Nishimura for pointing out 
a mistake in an earlier version of this paper and T. Panov for his comments.

\section{Cohomological condition}

Let $a,b$ be positive integers  and we fix them. 
Let $\gamma$ be the tautological line bundle over $\R P^{a-1}$
and let $\1$ denote a trivial real line bundle over an appropriate space.  
For a real vector bundle $E$, we denote by $P(E)$ the total space 
of the projectivization of $E$.  
For an integer $q$ such that $0\le q\le b$,  
we  set 
$$M(q):=P(q\gamma\oplus (b-q)\1).$$
Note that 
\begin{equation} \label{qb}
M(q)\text{ is diffeomorphic to } M(b-q)
\end{equation}
because $P(E\otimes L)$ and $P(E)$ are diffeomorphic 
for any smooth vector bundle $E$ and line bundle $L$ over a smooth manifold.  


A simple computation shows that 
\begin{equation} \label{HM}
H^*(M(q);\Z/2)=\Z/2[x,y]/(x^a,(x+y)^qy^{b-q})
\end{equation}
where $x$ is the pullback of the first Stiefel-Whitney class 
of $\gamma$ to $M(q)$ and 
$y$ is the first Stiefel-Whitney class of the 
tautological line bundle over $M(q)$. 
One easily sees that a set $\{x^iy^j\mid 0\le i< a,\ 0\le j< b\}$ 
is an additive basis of $H^*(M(q);\Z/2)$. 

\begin{lemm} \label{1}
If $0<q<b$, then both $y^a$ and $(x+y)^a$ are non-zero.
\end{lemm}

\begin{proof}
Suppose $y^a=0$.  Then it follows from \eqref{HM} that there are constants 
$c,d\in \Z/2$ and a homogeneous polynomial $f(x,y)$ in $x,y$ over $\Z/2$ 
such that 
\[
y^a=\begin{cases} cx^a \quad&\text{if $a< b$}\\
dx^a+f(x,y)(x+y)^qy^{b-q} \quad&\text{if $a\ge b$}
\end{cases}
\]
as polynomials in $x,y$.  
Clearly the former does not occur and the latter also does not occur 
because $q>0$ by assumption.  This is a contradiction, so $y^a\not=0$. 

If we set $X=x$ and $Y=x+y$, then $x+y=Y$ and $y=X+Y$, 
so that the role of $x$ and $x+y$ will be interchanged. 
Since $b-q>0$ by assumption, the above argument applied to 
$Y$ instead of $y$ proves that $(x+y)^{a}\not=0$.
\end{proof}

\begin{defi}
$h(a):=\min\{ n\in\N\cup \{0\}\mid 2^n\ge a\}.$
\end{defi}

For example, 
\[
\begin{split}
&h(1)=0,\ h(2)=1,\ h(3)=h(4)=2,\ h(5)=h(6)=h(7)=h(8)=3,\\ 
&h(9)=\dots=h(16)=4,\ \dots.
\end{split}
\]

\begin{lemm} \label{2}
Let $q$ and $q'$ be non-negative integers.  Then 
$\binom{q'}{i}\equiv \binom{q}{i} \pmod 2$ for $0\le \forall i < a$ if and only if 
$q'\equiv q \pmod{2^{h(a)}}$, where $\binom{n}{m}$ is understood to be $0$ 
when $n<m$ as usual. 
\end{lemm}

\begin{proof}
When $q'=q$, the lemma is trivial. We may assume that 
$q'> q$ without loss of generality. 
We note that the former congruence relations in the lemma are equivalent to the following 
congruence relation of polynomials in $t$ with $\Z/2$ coefficients:
\begin{equation} \label{cong}
(1+t)^{q'-q}\equiv 1 \pmod{t^a}.
\end{equation}

We shall prove the \lq\lq if" part first.  Suppose $q'\equiv q \pmod{2^{h(a)}}$. 
Then $q'-q=2^{h(a)}R$ with some positive integer $R$ and 
the left hand side of \eqref{cong} turns into 
\[
(1+t)^{q'-q}=(1+t^{2^{h(a)}})^R\equiv 1 \pmod{t^a}
\]
where the last congruence relation holds because $2^{h(a)}\ge a$.   This verifies  
\eqref{cong}. 

We shall prove the \lq\lq only if" part by induction on $a$. 
When $a=1$, $2^{h(a)}=1$ and hence the congruence relation $q'\equiv q\pmod{2^{h(a)}}$ trivially holds. 
Suppose that the induction assumption is satisfied for $a-1$ with $a\ge 2$ and that \eqref{cong} holds for $a$. 
Then \eqref{cong} holds for $a-1$, so the induction assumption implies 
$q'\equiv q \pmod{2^{h(a-1)}}$.  When $a-1$ is not a power of 2, 
$h(a-1)=h(a)$; so the congruence relation $q'\equiv q\pmod{2^{h(a)}}$ holds for $a$. 
When $a-1$ is a power of 2, say $2^s$, 
$$\text{$h(a-1)=s$,\quad  $h(a)=s+1$}$$
and $q'-q=2^sQ$ with some positive integer $Q$ because $q'\equiv q \pmod{2^{h(a-1)}}$. 
Therefore 
\[
(1+t)^{q'-q}=(1+t^{2^s})^Q=1+Qt^{2^s}+\text{higher degree terms}.
\]
Since this is congruent to $1$ modulo $t^a$ and $a>2^s=a-1$, $Q$ must be even. 
This shows that $q'\equiv q\pmod{2^{s+1}}$, proving the induction assumption for $a$ 
because $s+1=h(a)$. This completes the induction step and proves the 
\lq\lq only if" part. 
\end{proof}

\begin{theo} \label{main}
Let $0\le q,q'\le b$. Then 
$H^*(M(q);\Z/2)$ and $H^*(M(q');\Z/2)$ are isomorphic as graded rings if and only if
$q'\equiv q \text{ or } b-q \pmod{2^{h(a)}}$. 
\end{theo}

\begin{proof}
If both $q$ and $q'$ are in $\{0,b\}$, then the theorem is trivial.  
So we may assume $0<q<b$ without loss of generality.
We denote by $x'$ and $y'$ the generators in $H^*(M(q');\Z/2)$ corresponding to
$x$ and $y$. 

The \lq\lq if" part easily follows from \eqref{HM} and Lemma~\ref{2}.  
We shall prove the \lq\lq only if" part. 
Suppose that there is an isomorphism $$\varphi\colon H^*(M(q');\Z/2)\to H^*(M(q);\Z/2)$$ 
as graded rings. 
Since $\varphi(x')^a=\varphi({x'}^a)=0$, $\varphi(x')$ is neither $y$ nor $x+y$ 
by Lemma~\ref{1}.  Therefore $\varphi(x')=x$ and hence $\varphi(y')=y$ or $x+y$. 
Suppose $\varphi(y')=y$.  (When $\varphi(y')=x+y$, the role of 
$q$ and $b-q$ will be interchanged.)  Then $(x'+y')^{q'}{y'}^{b-q'}$ maps to 
$(x+y)^{q'}y^{b-q'}$ by $\varphi$ and it is zero in $H^*(M(q);\Z/2)$, so there 
are constants $c,d\in \Z/2$ and 
a homogeneous polynomial $f(x,y)$ in $x,y$ over $\Z/2$ such that 
\[
(x+y)^{q'}y^{b-q'}=\begin{cases} c(x+y)^qy^{b-q} \quad&\text{in case $a>b$}\\
d(x+y)^qy^{b-q}+f(x,y)x^a\quad&\text{in case $a\le b$}
\end{cases}
\]
as polynomials in $x,y$.  Clearly $c$ is non-zero, so $c=1$.  
Therefore $q'=q$ in case $a>b$. 
If $d=0$, then the right-hand side of the identity above in case $a\le b$ 
is divisible by $x$ as $a\ge 1$ 
while the left-hand side is not.  Therefore $d=1$ and 
the identity above in case $a\le b$ implies 
the former congruence relations in Lemma~\ref{2} by comparing the coefficients 
of $x^iy^{b-i}$ for $i<a$; so $q'\equiv q \pmod{2^{h(a)}}$ by Lemma~\ref{2}. 
\end{proof}

\begin{coro} \label{coro1}
Cohomological rigidity over $\Z/2$ holds for $M(q)$'s if $b\le 2^{h(a)}$. 
\end{coro}

\begin{proof} 
Suppose that $M(q)$ and $M(q')$ have isomorphic cohomology rings with 
$\Z/2$ coefficients. Then $q$ and $q'$ must satisfy the congruence 
relation in Theorem~\ref{main}.  But since $b\le 2^{h(a)}$, the congruence implies that $q'=q$ or $b-q$. 
This together with \eqref{qb} shows that $M({q'})$ is diffeomorphic to $M(q)$.  
\end{proof}

\section{KO theoretical condition}

In this section, we use KO theory to deduce a necessary and sufficient condition on $q$ and $q'$ 
for $M(q)$ and $M(q')$ to be diffeomorphic. 
We begin with a general lemma.  

\begin{lemm} \label{talong}
Let $E\to X$ be a real smooth vector bundle over a smooth manifold $X$. 
Let $\pi\colon P(E)\to X$ be the associated real projective bundle and let 
$\eta$ be the tautological real line bundle over $P(E)$.  
Then the tangent bundle $\tau P(E)$ of $P(E)$ with $\1$ added is isomorphic to 
$\Hom(\eta,\pi^*(E))\oplus \pi^*(\tau X)$.
\end{lemm}

\begin{proof}
A point $\ell$ of $P(E)$ is a line in $E$ and the fibers of $\eta$ 
over $\ell$ are vectors in the line $\ell$, so $\eta$ is a subbundle of $\pi^*(E)$.  
We give a fiber metric on $E$.  
It induces a fiber metric on $\pi^*(E)$ and we denote by $\eta^\perp$ 
the orthogonal complement of $\eta$ in $\pi^*(E)$. 
Then $\tau_f P(E)$ the tangent bundle along the fiber of 
$\pi\colon P(E)\to X$ is isomorphic to $\Hom(\eta,\eta^\perp)$.  
This is proved in \cite[Lemma 4.4]{mi-st74} when $X$ is a point and the same argument 
works for any $X$.  In fact, the argument is as follows. 
We note that the unit $S^0$ bundle $S(\eta)$ of $\eta$ can naturally be identified with 
the unit sphere bundle $S(E)$ of $E$.  
Let $v\in S(\eta)$ be in the fiber over $\ell\in P(E)$, that is, $v$ is a vector in 
the line $\ell$ with unit length.  
To an element $\psi\in \Hom(\eta,\eta^\perp)$ over $\ell\in P(E)$, 
we assign $\psi(v)$.  It is tangent to the fiber of $S(E)$ over $\pi(\ell)\in X$ 
at $v\in S(E)=S(\eta)$ and $\psi(-v)=-\psi(v)$, so 
$\psi(v)$ defines an element of $\tau_f P(E)$ over $\ell$. 
This correspondence gives an isomorphism from $\Hom(\eta,\eta^\perp)$ 
to $\tau_fP(E)$. 

Thus we obtain 
\begin{equation*}
\tau_f P(E)\oplus \1 \cong \Hom(\eta,\eta^\perp)\oplus \Hom(\eta,\eta)
\cong \Hom(\eta,\pi^*(E)).
\end{equation*}
This implies the lemma because $\tau P(E)\cong \tau_f P(E)\oplus\pi^*(\tau X)$. 
\end{proof}

\begin{defi}
$k(a):=\#\{ n\in\N \mid 0<n< a \text{ and } n\equiv 0,1,2,4\pmod{8}\}.$
\end{defi}

For example, 
\[
\begin{split}
&k(1)=0,\ k(2)=1,\ k(3)=k(4)=2,\ k(5)=k(6)=k(7)=k(8)=3,\\ 
&k(9)=4,\ k(10)=5,\ k(11)=k(12)=6,\dots
\end{split}
\]
It is known that $\widetilde{KO}(\R P^{a-1})$ is a cyclic group of order 
$2^{k(a)}$ generated by $\gamma-\1$ (\cite[Theorem 7.4]{adam62}). 
This implies that $2^{k(a)}\gamma$ is trivial because the fiber dimension (that is 
$2^{k(a)}$) is strictly larger than the dimension of the base space (that is $a-1$). 

\begin{theo} \label{main2}
Let $0\le q,q'\le b$. Then 
$M(q)$ and $M(q')$ are diffeomorphic if and only if $q'\equiv q\text{ or }b-q \pmod{2^{k(a)}}$.
\end{theo}

\begin{proof}
We shall prove the \lq\lq if" part first.  If $2^{k(a)}\ge b$ (this is the case when $a\ge b$), 
then $q'=q$ or $b-q$ and hence $M(q)\cong M(q')$ by \eqref{qb}.  Suppose $2^{k(a)}< b$.  
Then $a<b$ so that the bundles $q'\gamma\oplus (b-q')\1$ 
and  $q\gamma\oplus (b-q)\1$ are in the stable range and 
these bundles are isomorphic because $\widetilde{KO}(\R P^{a-1})$ 
is a cyclic group of order $2^{k(a)}$ generated by $\gamma-\1$ 
and $q'\equiv q\pmod{2^{k(a)}}$.  Hence $M(q)\cong M(q')$. 

We shall prove the \lq\lq only if" part. 
Suppose $M(q)\cong M(q')$ and let $f\colon M(q)\to M(q')$ 
be a diffeomorphism.  Then 
\[
f^*(\tau M(q'))=\tau M(q)\quad\text{in }\widetilde{KO}(M(q)).
\]
Since $\tau(\R P^{a-1})\oplus\1\cong a\gamma$, 
it follows from Lemma~\ref{talong} that the identity above implies 
\begin{equation} \label{f}
\begin{split}
f^*\big(\Hom(\eta',& q'\gamma\oplus (b-q')\1)\oplus a\gamma\big)\\
&=\Hom(\eta,q\gamma\oplus (b-q)\1)\oplus a\gamma \quad\text{in }\widetilde{KO}(M(q))
\end{split}
\end{equation}
where $\eta$ and $\eta'$ denote the tautological line bundles over $M(q)$ and $M(q')$ 
respectively and $\gamma$ is regarded as a line bundle over $M(q)$ and $M(q')$ 
through the projections onto $\R P^{a-1}$. 

If both $q$ and $q'$ are in $\{0,b\}$, then 
the \lq\lq only if" part is obviously satisfied.   
Therefore we may assume that $0<q<b$.  Then $f^*(x')=x$ and 
$f^*(y')=y$ or $x+y$ by Lemma~\ref{1}.  Therefore $f^*(\gamma)=\gamma$ 
and $f^*(\eta')=\eta$ or $\gamma\eta$. 
Suppose $f^*(\eta')=\eta$ occurs.  
(When $f^*(\eta')=\gamma\eta$ occurs, the role of $q$ and $b-q$ 
will be interchanged.) 
Then \eqref{f} reduces to 
\[
\Hom(\eta,q'\gamma\oplus (b-q')\1)
=\Hom(\eta,q\gamma\oplus (b-q)\1) \quad\text{in } \widetilde{KO}(M(q)).
\]
The fibration $M(q)\to \R P^{a-1}$ has a cross-section and we send 
the identity above to $\widetilde{KO}(\R P^{a-1})$ through the cross-section. Then 
$\eta$ becomes trivial or $\gamma$ because a line bundle over $\R P^{a-1}$ 
is either trivial or $\gamma$.  
In any case, the identity above reduces to 
\begin{equation} \label{qq}
(q'-q)(\gamma-\1)=0\quad\text{in }\widetilde{KO}(\R P^{a-1})
\end{equation}
and this implies $q'\equiv q \pmod{2^{k(a)}}$. 
\end{proof}

One easily sees that $h(a)\le k(a)$ for any $a$ and 
the equality holds if and only if $a\le 9$. 
Corollary~\ref{coro1} can be improved as follows. 

\begin{theo}
Cohomological rigidity over $\Z/2$ holds for $M(q)$'s if and only if 
$a\le 9$ or $b\le 2^{h(a)}$.   
\end{theo}

\begin{proof} 
If $a\le 9$, then $h(a)=k(a)$. So the \lq\lq if" part follows from 
Theorems~\ref{main} and ~\ref{main2} when $a\le 9$ and from 
Corollary~\ref{coro1} when $b\le 2^{h(a)}$. 

Suppose $a\ge 10$ (so $k(a)>h(a)\ge 4$) and $b>2^{h(a)}$. 
Then we take 
\[
(q,q')=\begin{cases} (1,2^{h(a)}+1) \quad&\text{when $b$ is a multiple of $2^{h(a)}$},\\
(0,2^{h(a)}) \quad&\text{when $b$ is not a multiple of $2^{h(a)}$}.
\end{cases}
\]
In both cases above, $q'\equiv q\pmod{2^{h(a)}}$ but 
$q'$ is not congruent to neither $q$ nor $b-q$ modulo $2^{k(a)}$ since 
$k(a)>h(a)\ge 4$. Therefore $M(q)$ and $M(q')$ are not diffeomorphic 
by Theorem~\ref{main2} while they have isomorphic cohomology 
rings with $\Z/2$ coefficients by Theorem~\ref{main}. 
\end{proof}

\section{Homotopical rigidity}

Cohomological rigidity over $\Z/2$ does not hold for $M(q)$'s in general, 
but the following holds. 

\begin{theo} \label{homot}
If $M(q)$ and $M(q')$ are homotopy equivalent, then they are 
diffeomorphic. 
\end{theo}

\begin{proof}
For a finite CW complex $X$, $J(X)$ denotes the $J$ group of $X$ and 
$J\colon \widetilde{KO}(X)\to J(X)$ denotes the $J$ homomorphism.  
Let $f\colon M(q)\to M(q')$ be a homotopy equivalence. Then 
\[
J(f^*(\tau M(q')))=J(\tau M(q))\quad\text{in $J(M(q))$}
\]  
by a theorem of Atiyah (\cite[Theorem 3.6]{atiy61}).  
The same argument as in the latter part of the proof of Theorem~\ref{main2} 
shows that we may assume that $0<q<b$ and $f^*(\eta')=\eta$, and then 
\[
J\big((q'-q)(\gamma-\1)\big)=0\quad\text{in }J(\R P^{a-1}).
\]
This implies \eqref{qq} because 
$J\colon \widetilde{KO}(\R P^{a-1})\to J(\R P^{a-1})$ is 
an isomorphism (see \cite[Theorem 13.9]{huse93}).  
Hence $M(q)$ and $M(q')$ are diffeomorphic. 
\end{proof}

Theorem~\ref{homot} motivates us to ask whether two real toric manifolds are 
diffeomorphic (or homeomorphic)  if they are homotopy equivalent, which we may call 
\emph{homotopical rigidity problem for real toric manifolds}. 


\end{document}